\documentclass[10pt, leqno]{amsart}
\usepackage{amsmath} 
\usepackage{amsthm}

\usepackage{amssymb}
\usepackage{xy}
\usepackage{xypic}
\usepackage{pifont}
\usepackage{stmaryrd}

\usepackage[usenames, dvipsnames]{color}

\theoremstyle{plain}

\newcommand{\zig}{\addtocounter{Lem}{1}\tag{\theLem}}

\newcommand{\dotcup}{\ensuremath{\mathaccent\cdot\cup}}

\def\:{\colon} 
\DeclareMathOperator*{\colim}{colim}
\DeclareMathOperator*{\holim}{holim}

\newtheorem{Lem}{Lemma}[section]

\newtheorem{Cor}[Lem]{Corollary}
\newtheorem{Thm}[Lem]{Theorem}

{\theoremstyle{definition} % non-italicized text in theorem environment
\newtheorem{Ex}[Lem]{Example}

\newtheorem{Rk}[Lem]{Remark}

\newtheorem{Def}[Lem]{Definition}}
\begin{document} 
 
\title[Lifting group cohomology to spectra for trivial profinite 
actions]{A lift from group cohomology to spectra for trivial profinite 
actions} 

\author{Daniel G. Davis}
 \maketitle

\begin{abstract}
Let $G$ be a profinite group, $X$ a discrete $G$-spectrum with 
trivial action, and $X^{hG}$ the continuous homotopy fixed points. 
For any $N \trianglelefteqslant_o G$ (``$o$" for open), 
$X = X^N$ is a $G/N$-spectrum with trivial action. We construct a zigzag $\colim_N X^{hG/N} \xrightarrow{\,\Phi\,} \colim_N (X^{hN})^{hG/N} \xleftarrow{\,\Psi\,} X^{hG}$, where 
$\Psi$ is a weak equivalence. When $\Phi$ is a weak equivalence, 
this zigzag gives an interesting model for $X^{hG}$ (for example, its 
Spanier-Whitehead dual is $\holim_N F(X^{hG/N}, S^0)$). We prove that this happens in the following cases: (1) $|G| < \infty$; (2) $X$ is bounded 
above; (3) there exists $\{U\}$ cofinal in $\{N\}$, such that for each $U$, $H^s_c(U, \pi_\ast(X)) = 0$, 
for $s > 0$. Given (3), for each $U$, there is a weak equivalence $X \xrightarrow{\,\simeq\,} X^{hU}$ and $X^{hG} \simeq X^{hG/U}$. For case (3), we give a series of corollaries and examples. 
As one instance of a family of examples, if $p$ is a prime, $K(n_p,p)$ 
the $n_p$th Morava $K$-theory $K(n_p)$ at $p$ for some $n_p \geq 1$, 
and $\mathbb{Z}_p$ the $p$-adic integers, then for each $m \geq 2$, 
%all the above conclusions hold 
(3) is satisfied when $G \leqslant \prod_{p \leq m} 
\mathbb{Z}_p$ 
%\mathbb{Z}_2 
%\times \mathbb{Z}_3 \times \mathbb{Z}_5$ 
is closed, $X = \bigvee_{p > m} (H\mathbb{Q} \vee K(n_p,p))$,
%\bigl(\bigvee_{i = 1}^\infty H\mathbb{Q}\bigr) \vee \bigl(\bigvee_{p > m} K(n_p,p)\bigr)$, 
and $\{U\} := \{N_G \mid N_G  \trianglelefteqslant_o G\}$.
\end{abstract}

\section{Introduction}\label{intro_proper}

\subsection{The problem considered in this work and our main 
results}\label{intro} In this paper, ``spectrum" always refers to an object in $Sp^\Sigma$, the stable model category of symmetric spectra of 
simplicial sets \cite{HSS}.

Let $G$ be any profinite group and let $X$ be a discrete $G$-spectrum 
(in the sense of \cite{joint}), such that the $G$-action is trivial (thus, any spectrum, equipped with the trivial $G$-action, can be regarded as such a discrete 
$G$-spectrum). 
Let $\{N\}$ be the collection of open normal subgroups 
of $G$. For each $N$, $G/N$ is a finite group and $X = X^N$ can be regarded as a $G/N$-spectrum having trivial 
$G/N$-action, and hence, the homotopy fixed point spectrum 
\[X^{hG/N} := (X^N)^{hG/N}\] is defined. Then there is a map
\[\Phi \: \colim_N X^{hG/N} \rightarrow \colim_N (X^{hN})^{hG/N},\] where the colimits are indexed by the collection $\{N\}$, and 
a weak equivalence 
\[\colim_N (X^{hN})^{hG/N} \xleftarrow[\,\textstyle{\simeq}\,]{\ \ 
\textstyle{\Psi} \ \ } X^{hG}.\] Here, by definition (see \cite[Section 3.1]{joint}), $X^{hG}$ (and, similarly, $X^{hN}$) is 
the continuous homotopy fixed point spectrum
\[X^{hG} := (X_{fG})^G,\] where $\theta_{\scriptscriptstyle{G}} \: X \xrightarrow{\,\simeq\,} X_{fG}$ is a functorial fibrant replacement in the model category $\Sigma\mathrm{Sp}_G$ of discrete $G$-spectra (hence, this map is a trivial cofibration in  $\Sigma\mathrm{Sp}_G$). The construction of 
$\Phi$ and $\Psi$ is 
explained in Section \ref{map_phi} (strictly speaking, the source and target of $\Phi$ involve identifications, which are explained in Section \ref{map_phi}). 

It is natural to ask when $\Phi$ is a weak equivalence, 
since this implies that there is an equivalence 
\begin{equation}\label{corollary}\zig
X^{hG} \simeq \colim_N X^{hG/N} 
\end{equation} 
that expresses 
the homotopy fixed points for 
a typically infinite group in terms of homotopy fixed points 
for finite groups. 

To point out another interesting feature of (\ref{corollary}), we note 
the following. 
If $M$ is a discrete $G$-module and $H$ is a closed subgroup of 
$G$, then we let $H^\ast_c(H, M)$ denote the continuous cohomology of $H$, with coefficients in $M$, where $M$ is regarded as a discrete $H$-module. 
% new insertion 
%Throughout this Introduction and everywhere else in this paper, 
%Also, given any spectrum $Z$ and $t \in \mathbb{Z}$, 
%by $\pi_t(Z)$, 
%we mean the homotopy groups $[S^t, Z]$ of morphisms in the 
%homotopy category of $Sp^\Sigma$, where here, $S^t$ denotes a 
%fixed cofibrant and fibrant model for the $t$-th suspension of the sphere spectrum. 
% end of new insertion
For each $t \in \mathbb{Z}$, the stable homotopy group $\pi_t(X)$ is naturally a discrete $G$-module with trivial $G$-action. For any 
$N$, $(\pi_t(X))^N = \pi_t(X)$, so that every $\pi_t(X)$ can be 
regarded as a $G/N$-module with trivial action. Also, if $K$ is a finite group and $Y$ is a (naive) $K$-spectrum, then 
$K$, with the discrete topology, is a profinite group 
and $Y$ can be regarded as a discrete $K$-spectrum, and 
the usual homotopy fixed point spectrum $\mathrm{Map}_\ast(EK_+, Y)^K$ and the continuous homotopy fixed point spectrum $Y^{hK}$ 
are equivalent to each other. 

Since the homotopy fixed points 
$(-)^{hG}$ and each $(-)^{hG/N}$ are 
the right derived functors of the right Quillen functors 
\[(-)^{G} \: \Sigma\mathrm{Sp}_G \to Sp^\Sigma \ \ \ \text{and each} 
\ \ \ (-)^{G/N} \: \Sigma\mathrm{Sp}_{G/N} \to Sp^\Sigma\] 
\cite[Section 3.1]{joint}, 
respectively, the equivalence in (\ref{corollary}) (when it holds) 
is a lift from 
algebra to homotopy theory of the 
well-known isomorphism
\[H^*_c(G, \pi_t(X)) \cong \colim_N H^*(G/N, \pi_t(X)), \ \ \ \text{all} \ 
t \in \mathbb{Z}\] 
(which is valid for any $X$ in $\Sigma\mathrm{Sp}_G$ with trivial 
action), for $H^\ast_c(G, -)$ and each $H^\ast(G/N, -)$ are the 
right derived functors of the familiar left exact functors 
$(-)^G$ and $(-)^{G/N}$, respectively, that land in abelian groups.   

Now we come to our work on $\Phi$. In Theorem \ref{WhenFinite}, we 
show that $\Phi$ is a weak equivalence whenever $G$ is finite. But 
as mentioned earlier, in the world of continuous actions by profinite 
groups, we usually are mostly interested in the case where $G$ is 
infinite: for possibly infinite profinite groups, we have two main 
results, Theorems \ref{coconn} and \ref{vanishing} below. Also, 
for Theorem \ref{vanishing}, we give -- in Section \ref{1.2} -- several corollaries and a few examples. As a quick sample of these, we 
note that in Example \ref{multipleK(n)s}, we prove, for instance, 
the following, where given $n \geq 0$ and a prime $p$, 
$K(n,p)$ denotes the $n$th Morava $K$-theory spectrum $K(n)$ at $p$, $\mathbb{Z}_p$ is the $p$-adic integers, and, for each $i \geq 1$, 
$p_i$ is the $i$th prime (that is, $p_1 = 2$, $p_2 = 3$, etc.): for 
\[X = \textstyle{\bigl(\bigvee_{i = 1}^\infty K(0, p_i)\bigr) \vee \bigl(\bigvee_{i = 1}^\infty 
K(n_i, p_{2i-1})\bigr)},\] with each $n_i \geq 1$, and 
$G = \prod_{i = 1}^\infty \mathbb{Z}_{p_{2i}},$ 
with $G$ acting trivially on $X$, $H$ any closed subgroup of $G$, and $N_H$ any open normal subgroup of $H$, there is a map
\[X \xrightarrow{\,\simeq\,} X^{hH} \simeq X^{hH/{N_H}}\] that, in addition to $\Phi$, is 
a weak equivalence, where, here, $\Phi$ is defined relative to $H$ (so that ``$N$" is an open normal subgroup of $H$).  

In work underway with Thomas Credeur, we are 
considering examples in chromatic homotopy theory of when the map 
$\Phi$ is not a weak equivalence. In Section \ref{1.3}, 
after describing two consequences of (\ref{corollary}), 
we explain that part of our motivation for studying $\Phi$ comes 
from examples in joint work of Beaudry, Goerss, Hopkins, and 
Stojanoska of homotopy fixed point spectra for the trivial action of 
certain profinite groups on the $K(n)$-local sphere. 

To help state our first main result, we 
recall that a spectrum $Z$ is {\em bounded above} if there exists a fixed integer $t_0$ such that $\pi_t(Z) = 0$, for all $t > t_0$. For example, if 
$Z$ is bounded above and $t_0$ can be taken to be $0$, then $Z$ is 
said to be {\em coconnective}. Applications of coconnective spectra 
are given, for instance, in \cite[Propositions 3.3, 3.9, 3.14]{DwyerEtAl} 
and \cite[Section 4]{LawsonHinfinity}.

\begin{Thm}\label{coconn}
If $G$ is any profinite group and $X$ is a discrete $G$-spectrum with 
trivial $G$-action, such that $X$ is bounded above, 
then the map $\Phi$ 
is a weak equivalence and 
$X^{hG} \simeq \colim_N X^{hG/N},$ where the colimit is indexed 
over all the open normal subgroups of $G$. 
\end{Thm}

In Section \ref{section_4}, after giving a general tool (Theorem \ref{diagram}) that helps with deciding when $\Phi$ is a weak equivalence, we give the proof of Theorem \ref{coconn}.

\begin{Thm}\label{vanishing}
Let $G$ be a profinite group and let $X$ be a discrete $G$-spectrum with trivial $G$-action. Suppose that $\{U\}$ is a collection of open normal subgroups of $G$ that is cofinal in $\{N\}$, the collection of 
all open normal subgroups of $G$, such that 
for each $U$ and every $t \in \mathbb{Z}$, $H^s_c(U, \pi_t(X)) = 0$, for all $s > 0$. 
Then each of the following holds: 
\begin{enumerate} 
\item[(a)]
the map $\Phi$ is a weak equivalence$\mspace{1mu}\mathrm{;}$ 
\item[(b)]
there is an equivalence 
$X^{hG} \simeq \colim_N X^{hG/N};$ and
\item[(c)] 
for each $U,$ there is $\mathrm{(1)}$ an equivalence
$X^{hG/U} \simeq X^{hG}$ and
\item[{}]
% for each $U,$ 
$\mathrm{(2)}$ a weak equivalence 
$X \xrightarrow{\,\simeq\,} X^{hU}.$
\end{enumerate}
\end{Thm}

For this result, we only need to prove 
(a) and (c), and this is done in Section 
\ref{section_vanishing}. 
%By (\ref{corollary}), conclusion (b) in Theorem 
%\ref{vanishing} is a consequence of (a), and the proof of the remaining 
%parts of this result is in Section \ref{section_vanishing}. 

% some new text
In Theorem \ref{vanishing}, suppose that $G$ is finite: above, we 
noted that this hypothesis and the theorem's first sentence are 
enough to yield conclusions (a) and (b); if $U'$ is any normal 
subgroup of $G$ with $H^s(U', \pi_\ast(X)) = 0$ for $s > 0$, then 
it is well-known that the homotopy 
fixed point spectral sequence 
\[H^s(U', \pi_t(X)) \Longrightarrow \pi_{t-s}(X^{hU'})\] 
collapses, with 
\[\pi_\ast(X^{hU'}) \cong H^0(U', \pi_\ast(X)) = (\pi_\ast(X))^{U'} = \pi_\ast(X),\] so that 
$X \simeq X^{hU'}$; also, there is an 
equivalence $X^{hG} \simeq (X^{hU'})^{hG/U'}$ (for example, see 
\cite[Lemma 10.5]{DwyerWilkerson} and \cite[Proposition 3.3.1]{joint}). Thus, 
when $G$ is finite, the conclusions of Theorem \ref{vanishing}  
are not surprising. However, as mentioned earlier, in this result -- and in 
its corollaries and examples -- we are mostly interested in the 
case when $G$ is infinite, and we do not mean to claim that the 
conclusions given by our results when $G$ is finite originate with this paper. The above spectral sequence argument 
is used in the proof of Theorem \ref{vanishing}, but when $G$ is infinite, it 
takes some work to know that the homotopy fixed point 
spectral sequence exists. 
% end of some new text 

\subsection{Corollaries and examples of the second main result}\label{1.2}

\begin{Cor}\label{K(0)corollary} 
If $G$ is any profinite group and $X$ is a discrete $G$-spectrum with trivial $G$-action, such that for each $t \in \mathbb{Z}$, 
$\pi_t(X)$ is a torsion-free divisible abelian group, then $\Phi$ is a 
weak equivalence and there is an equivalence $X^{hG} \simeq X^{hG/N}$ and a weak equivalence 
%\[X^{hG} \xleftarrow{\,\simeq\,} 
$X \xrightarrow{\,\simeq\,} X^{hN},$ 
where $N$ is any open normal subgroup of $G$. 
\end{Cor} 

\begin{proof}
For all $N$ and $t \in \mathbb{Z}$, $H^s_c(N, \pi_t(X)) = 0$, when $s > 0$, by \cite[Corollary 6.7.5]{Ribes}. Hence, 
the hypotheses of Theorem \ref{vanishing} are satisfied with 
``$\{U\}$" equal to $\{N\}$.
\end{proof} 

\begin{Ex}\label{K(0)example}
Let $H\mathbb{Q}$ be the Eilenberg-Mac Lane spectrum for the rationals $\mathbb{Q}$ and let $G$ be a profinite group. Also, 
let $X$ be an $H\mathbb{Q}$-module (in the sense of \cite[Section 5.4]{HSS}) that is regarded as a 
discrete $G$-spectrum with trivial action. Then for 
every integer $t$, $\pi_t(X)$ is a unitary $\mathbb{Q}$-module, so that 
its underlying abelian group is torsion-free divisible. Thus, the 
conclusions of Corollary \ref{K(0)corollary} 
hold, giving that for every profinite 
group $G$, there is a weak equivalence 
$X \xrightarrow{\,\simeq\,} X^{hG}$. In particular, there is a 
weak equivalence 
$H\mathbb{Q} \xrightarrow{\,\simeq\,} (H\mathbb{Q})^{hG}$.
\end{Ex}

% relation to the literature
We point out that if $G$ is a finite group and 
$M$ is a simply connected rational $G$-space, then 
relationships between $M$, $M^{hG}$, and 
$(\pi_\ast(M))^G$ are studied in the thesis \cite{Goyo} (see also \cite[Appendix A]{BuijsEtAl}). For a result about simply connected pointed simplicial discrete $G$-sets, where $G$ is an arbitrary profinite group, rational localization, and fixed points, see \cite[Theorem 8.1 and the sentence after it]{hGal} (the second part of this citation considers the case when the $G$-action is trivial). 
% end of relation to the literature

\begin{Rk}\label{HQ-module}
Let $X$ be as in Corollary \ref{K(0)corollary} and let $\eta \: S^0 \to 
H\mathbb{Q}$ be the unit of $H\mathbb{Q}$. Then $\pi_\ast(X)$ is a 
$\mathbb{Q}$-vector space in each degree and it is a familiar fact 
that this implies that the composition
\[X \xrightarrow{\,\cong\,} S^0 \wedge X \xrightarrow{\,\eta \wedge \mathrm{id}_X\,} H\mathbb{Q} \wedge X\] 
is a weak equivalence of spectra (for example, see the argument in 
\cite[proof of Corollary 2.16: 2nd paragraph]{ShipleyHZ}). Thus, $X$ is weakly equivalent to the 
$H\mathbb{Q}$-module $H\mathbb{Q} \wedge X$.  
\end{Rk}

% tying things together
We tie together Corollary \ref{K(0)corollary}, 
Example \ref{K(0)example}, and Remark \ref{HQ-module} with the 
following immediate result. 
\begin{Cor}\label{ties_together}
If $X$ is a spectrum with $\pi_\ast(X)$ degreewise torsion-free divisible and $G$ is any profinite group, then when $X$ and $H\mathbb{Q} \wedge X$ are regarded as discrete $G$-spectra with trivial $G$-action, the map $X \xrightarrow{\,\simeq\,} H\mathbb{Q} \wedge X$ is a weak equivalence in $\Sigma\mathrm{Sp}_G$ and there are weak equivalences 
\[X \xrightarrow{\,\simeq\,} X^{hG} \xrightarrow{\,\simeq\,} 
(H\mathbb{Q} \wedge X)^{hG} \xleftarrow{\,\simeq\,} H\mathbb{Q} \wedge X.\]
\end{Cor}
% end of tying things together

Now we take Theorem \ref{vanishing} in a different direction, where the 
primary focus is on torsion in $\pi_\ast(X)$. 
Let $J$ be a nonempty subset ($J$ can be infinite) 
of the primes. Following standard practice, we say that 
an abelian group $A$ is a {\em $J$-torsion group} if each nonzero 
element $a \in A$ 
has finite order $n_a$, such that if $p$ is a prime that divides $n_a$, then $p \in J$. The {\em order} of a profinite group $G$, $\#G$, is the supernatural number defined as the least common multiple of the 
set $\{\mspace{1mu} |G/{N'}| \mid N' \in \{N\}\}$ (see, for example, 
\cite[Section 2.3]{Ribes}). Given $G$, we write the map ``$\Phi$" as ``$\Phi_G$" 
when we need to emphasize the role of $G$ as the ``ambient group" for 
the construction of $\Phi$. 

% i commented out the initial version of this.

\begin{Cor}\label{set_primes}
Let $J$ be any nonempty subset of the 
primes and let $G$ be a profinite 
group, such that if $p \in J$, then $p$ does not divide $\# G$. 
Also, let $X$ be a discrete $G$-spectrum with trivial action, such that for each $t \in \mathbb{Z}$, $\pi_t(X)$ is isomorphic as an abelian group to the direct sum of a torsion-free divisible abelian group and a $J$-torsion group. Then for any closed subgroup $H$ of $G$, 
the map $\Phi_H$ is a weak equivalence and 
there are equivalences
\[\colim_{N_H} X^{hH/{N_H}} \simeq X^{hH} \xleftarrow{\,\simeq\,} 
X \xrightarrow{\,\simeq\,} X^{hN'_H} \simeq X^{hH/{N'_H}},\] where the colimit is indexed over 
the collection $\{{N_H}\}$ of all open normal subgroups of $H$
% all the open normal subgroups $N_H$ of $H$ 
and ${N'_H}$ is any particular open normal subgroup of $H$.
\end{Cor}

\begin{proof}
Choose any ${N_H}$, $s \geq 1$, and $t \in \mathbb{Z}$. 
There are discrete $G$-modules with trivial $G$-action, $D$ and $T$, 
torsion-free divisible and $J$-torsion, respectively, such that 
\[H^s_c(N_H, \pi_t(X)) \cong 
H^s_c(N_H, D \oplus T)
\cong H^s_c(N_H, D) \oplus H^s_c(N_H, T) 
\cong H^s_c(N_H, T).\] 
% since $H^s_c(N_H, D) = 0$. 
For $p \in J$, since $p \nmid \#G = [G : N_H](\#N_H)$, $p \nmid \#N_H$. By \cite[Lemma 10.2.1]{Wilson}, the map $H^s_c(N_H,T) \xrightarrow{\,\mathrm{res}\,} H^s_c(1,T) = 0$ is injective, giving $H^s_c(N_H,\pi_t(X)) = 0$. Thus, Theorem \ref{vanishing} applies 
% the hypotheses of Theorem \ref{vanishing} hold 
with ``$G$" 
and ``$\{U\}$" equal to $H$ and $\{{N_H}\}$, 
 % the collection $\{{N_H}\}$ of all open normal subgroups of $H$, 
respectively.
\end{proof}

\begin{Ex}\label{multipleK(n)s}
Recall from Section \ref{intro} that if $p$ is a prime and $n \geq 0$, we let $K(n,p)$ denote the $n$th 
Morava $K$-theory spectrum (usually written as $K(n)$): $K(0, p) = H\mathbb{Q}$ and for $n \geq 1$, $\pi_\ast(K(n,p)) = \mathbb{F}_p[v_n, v_n^{-1}]$, where $\mathbb{F}_p$ is 
the field with $p$ elements and the degree of $v_n$ is $2(p^n-1)$. 
Now let there be the following sets: $P$, the set of all primes; $J$, a nonempty proper subset of $P$; $L$, a nonempty subset of $P-J$; and 
$P'$, a subset of $P$ that can be empty. We define $J' := P' \,\dotcup\, J$ 
as the disjoint union of $P'$ and $J$: if $p \in P' \cap J$, then $p$ appears twice, as distinct elements, in $J'$. 
%($J$, $L$, and $P'$ can be infinite). 
Let \[G = \textstyle{\prod_{\ell \in L} \mathbb{Z}_\ell,}\] where $\mathbb{Z}_\ell$ is the $\ell$-adic integers. Since $\#G = \prod_{\ell \in L} \ell^\infty$, if $p \in J$, then $p$ does not divide $\#G$. For each $p \in J'$, if $p \in P'$, 
let $n_p = 0$, and if $p \in J$, let $n_p$ be a positive integer (if $p \in P' \cap J$, then $n_p$ for $p \in P'$ and $n_p$ for $p \in J$ are distinct 
elements in $\{n_q \mid q \in J'\}$). Let \[\textstyle{X = \bigvee_{p \in J'} K(n_p, p)}\] be a discrete $G$-spectrum with trivial action. 
% Let $J' = \{p \mid p \in J, \ n_p = 0\}$. 
For each integer $t$, 
there is the isomorphism $\pi_t(X) \cong D_t \oplus T_t$, where 
\[D_t = \bigoplus_{p \in P'} \pi_t(K(0, p)) 
= \begin{cases} \textstyle{\bigoplus_{p \in P'} \mathbb{Q}}, & \text{if} \ t=0;\\ \{e\}, & \text{if} \ t \neq 0\end{cases}
\ \ \ \ \mathrm{and} \ \ \ T_t = \bigoplus_{p \in J} \pi_t(K(n_p, p))\] are torsion-free divisible and $J$-torsion, respectively (if $P' = \varnothing$, 
then %the corresponding 
$D_t 
%or $T_t$, respectively, 
= \{e\}$). Then $J$, $G$, and $X$ satisfy the hypotheses of Corollary \ref{set_primes}, so that, 
for example, there are weak equivalences
\[\textstyle{\bigl(\mspace{1mu}\bigvee_{p \in J'} K(n_p, p)\bigr)^{h(\prod_{\ell \in L} \mathbb{Z}_\ell)} 
\xleftarrow{\,\simeq\,} \bigvee_{p \in J'} K(n_p, p) \xrightarrow{\,\simeq\,} \bigl(\mspace{1mu}\bigvee_{p \in J'} K(n_p, p)\bigr)^{hH}},\] where $H$ is any closed 
subgroup of $G$, and, as the special case 
for when $p$ and $\ell$ are distinct primes and $n$ is any positive integer, by setting $J = J' = \{p\}$, $L = \{\ell\}$, 
and $n_p = n$, we find that there is a weak equivalence \[K(n) \xrightarrow{\,\simeq\,} 
K(n)^{h\mathbb{Z}_\ell}.\] 
\end{Ex}

% a new example
In the above applications of Theorem \ref{vanishing}, 
we always let ``$\{U\}$" in the hypotheses of the theorem be equal to the entire collection $\{N\}$ of open 
normal subgroups of ``$G$" (for example, in the proof of 
Corollary \ref{set_primes}, ``$G$" and ``$\{U\}$" equal $H$ and 
$\{N_H\}$, respectively). Below, we give an example where ``$\{U\}$" 
must be a proper subset of $\{N\}$. % some details 
To give this example, we need the following notation and 
recollection. 

Given a spectrum $Z$, let the morphism 
$Z \xrightarrow{\,\simeq\,} Z_f$ be a functorial fibrant replacement 
in $Sp^\Sigma$: this map is a trivial cofibration and $Z_f$ is fibrant, 
in $Sp^\Sigma$. If $W$ is a simplicial set, 
then we let $F(W_+, Z)$ denote the power 
spectrum $Z^{W_+}$ of \cite[Section 1.3]{HSS}, where $W_+$ is 
$W$ with a disjoint basepoint added. 
By \cite[pages 163-165]{HSS}, $F(W_+, Z)$ is the 
function spectrum $\mathrm{Hom}_S(F_0(W_+), Z)$, which is 
often written as $F(\Sigma^\infty W_+, Z)$. Then if 
$K$ is a finite group and 
$Z$ is regarded as a $K$-spectrum with trivial $K$-action, we recall 
(for example, see \cite[10.2]{DwyerWilkerson}) that 
there is an equivalence
\[Z^{hK} \simeq F(BK_+, Z_f),\] where the simplicial set $BK$ is the 
classifying space of $K$.

% end of some details
\begin{Ex}\label{proper_sub}
Let $p$ and $\ell$ be distinct primes, with $n, r \geq 1$, and 
set \[G = 
\mathbb{Z}_\ell \times (\mathbb{Z}/{(p^r\mathbb{Z})})\] and 
$X = K(n)$, 
regarded as a discrete $G$-spectrum with trivial action. Also, 
for each $m \geq 0$, let $\ell^m\mathbb{Z}_\ell$ 
denote the open normal subgroup $(\ell^m\mathbb{Z}_\ell) \times \{e\}$ of $G$ 
and set \[\{U\} = \{\ell^m\mathbb{Z}_\ell \mid m \geq 0\}.\] Again, let 
$m \geq 0$: since 
$\pi_\ast(K(n))$ is degreewise $\{p\}$-torsion, 
$\ell^m\mathbb{Z}_\ell \cong \mathbb{Z}_\ell$, and 
$p$ does not divide $\#(\ell^m\mathbb{Z}_\ell) = \ell^\infty$, 
\cite[Lemma 10.2.1]{Wilson} gives that 
\[H^s_c(\ell^m\mathbb{Z}_\ell, \pi_\ast(K(n))) = 0, \ \ \ \text{for all} 
\ s > 0.\] Thus, $G$, $X$, and $\{U\}$ satisfy the hypotheses of Theorem 
\ref{vanishing}, so that, for example,
\[K(n)^{hG} \simeq 
K(n)^{h((\mathbb{Z}_\ell \times (\mathbb{Z}/{(p^r\mathbb{Z})}))/((\ell^0\mathbb{Z}_\ell) \times \{e\}))},\] which 
yields the equivalence 
\[K(n)^{h(\mathbb{Z}_\ell \times (\mathbb{Z}/{(p^r\mathbb{Z})))}}
\simeq K(n)^{h(\mathbb{Z}/{(p^r\mathbb{Z})})}.\] 
Furthermore, since 
\begin{align*} 
\pi_{-\ast}(K(n)^{h(\mathbb{Z}/{(p^r\mathbb{Z})})}) & \cong 
\pi_{-\ast}(F(B(\mathbb{Z}/{(p^r\mathbb{Z})})_+, K(n)_f)) 
\cong K(n)^{\ast}(B(\mathbb{Z}/{(p^r\mathbb{Z})})) 
\\ & \cong K(n)^\ast[x]/(x^{p^{nr}}),
\end{align*} 
where 
$K(n)^\ast = \pi_{-\ast}(K(n))$ and $x \in K(n)^2(\mathbb{C}P^{\infty})$, 
is a free $K(n)^\ast$-module of rank $p^{nr} > 1$ (when $p$ is 
odd, see \cite[Theorem 5.7; page 745]{RavenelWilson}; for $p = 2$, 
see \cite{RavenelFiniteGroups} and \cite[Theorem 1.2]{KuhnTransactions}; for all primes $p$, the succinct presentation 
of this computation in \cite[Proposition 2.3]{SanFeliu} is helpful), 
$\pi_\ast(K(n)^{h(\mathbb{Z}/{(p^r\mathbb{Z})})})$ and 
$\pi_\ast(K(n))$ are not isomorphic as graded abelian groups, and hence, 
there is no equivalence between $K(n)^{hG}$ and $K(n)$. Thus, if the 
hypotheses of Theorem \ref{vanishing} hold for $G$, $X$, and 
some $\{U\}$, where here, $\{U\}$ need not be the collection 
$\{\ell^m\mathbb{Z}_\ell \mid m \geq 0\}$ used above, $G$ 
does not belong to $\{U\}$.  
\end{Ex}
% end of a new example
\subsection{Two consequences of (\ref{corollary}), our motivation for studying $\Phi$, and some technical comments}\label{1.3} Now let $G$ be any profinite group, let $X$ be any object 
in $\Sigma\mathrm{Sp}_G$ that has trivial $G$-action, and suppose that $\Phi$ is a weak equivalence. 
Then an interesting consequence of the equivalence 
$X^{hG} \simeq \colim_N X^{hG/N}$ from (\ref{corollary}) is that it implies that 
the Spanier-Whitehead dual of 
$X^{hG}$, $D(X^{hG})$, satisfies
\[D(X^{hG}) = F(X^{hG}, S^0) \simeq \holim_N F(X^{hG/N}, S^0),\] 
%= \holim_N D(X^{hG/N}),\] where the 
%right-hand side is the homotopy limit of the Spanier-Whitehead dual 
%of each $X^{hG/N}$. 
where here, $S^0$ is a fixed cofibrant and fibrant model for the sphere 
spectrum. 

We continue to make the assumptions of the previous paragraph, to 
point out another form of the equivalence in (\ref{corollary}). Let the morphism
\[\eta \: X \xrightarrow{\,\simeq\,} X_f\] be a functorial fibrant replacement 
in $Sp^\Sigma$. 
%: this map is a trivial cofibration of 
%spectra and $X_f$ is a fibrant spectrum. If $W$ is a simplicial set and 
%$Z$ is a spectrum, then we let $F(W_+, Z)$ denote the power 
%spectrum $Z^{W_+}$ of \cite[Section 1.3]{HSS}, where $W_+$ is 
%formed by adding a disjoint basepoint to $W$. 
%By \cite[pages 163-165]{HSS}, $F(W_+, Z)$ is the 
%function spectrum $\mathrm{Hom}_S(F_0(W_+), Z)$, which is 
%often written as $F(\Sigma^\infty W_+, Z)$. 
Since when $K$ is a finite group and 
$X$ is also a $K$-spectrum with trivial action, 
there is an equivalence
$X^{hK} \simeq F(BK_+, X_f),$ the equivalence in (\ref{corollary}) 
implies that 
\begin{equation}\label{colimit}\zig
X^{hG} \simeq \colim_N F(B(G/N)_+, X_f),
\end{equation} 
giving an explicit 
presentation of $X^{hG}$ in terms of fundamental building blocks. 

% motivation
Now we describe part of our motivation for studying $\Phi$. Let 
$n \geq 1$, let $p$ be any prime, set $\mathbb{S}_n$ equal 
to the $n$th Morava stabilizer group, and let 
\[\mathbb{S}_n = \Gamma_0 \gneq \Gamma_1 \gneq \cdots 
\gneq \Gamma_m \gneq \Gamma_{m+1} \gneq \cdots\] be a cofinal collection of open 
normal subgroups of $\mathbb{S}_n$. 
In the talk ``Equivariant dual of Morava $E$-theory" about joint work with Agn\`{e}s Beaudry, Paul Goerss, 
and Vesna Stojanoska (this talk was July 18, 2017, at the University 
of Illinois at Urbana-Champaign; a video of this talk is available online 
at YouTube), Mike Hopkins considered ``$(L_{K(n)}S^0)^{h\Gamma_m}$," the homotopy fixed points of the $K(n)$-local sphere 
$L_{K(n)}S^0$ with respect to the trivial action of $\Gamma_m$, 
and Hopkins wrote this object as ``$D(B{\Gamma_m}_+)$," which, 
when this notation is taken literally in the context of the talk, is 
equal to $F(\Sigma^\infty B{\Gamma_m}_+, L_{K(n)}S^0)$, the $K(n)$-local Spanier-Whitehead dual of $\Sigma^\infty B{\Gamma_m}_+$ (we are not saying that this definition of $D(B{\Gamma_m}_+)$ is 
what was meant by Hopkins). (For part of what is discussed in this paragraph and the next one, a slightly different 
presentation from the one given by Hopkins in his talk is in 
\cite[pages 54-55]{BarthelBeaudry}.)

Let $\mathbb{G}_n$ be the extended Morava stabilizer group and let 
$E_n$ be the Lubin-Tate spectrum (see \cite{DH}). In his talk, 
Hopkins 
explained how the $D(B{\Gamma_m}_+)$'s play a role in better 
understanding the $\mathbb{G}_n$-action on 
$F(E_n, L_{K(n)}S^0)$ that comes from the $\mathbb{G}_n$-action on $E_n$ (see \cite[Proposition 16]{Str}). It is the author's understanding that the construction ``$(L_{K(n)}S^0)^{h\Gamma_m}$" is 
not formed by regarding 
$L_{K(n)}S^0$ as a discrete $\Gamma_m$-spectrum with 
trivial action and then applying $((-)_{f\Gamma_m})^{\Gamma_m}$ 
to it, and so 
further consideration of this construction is outside the scope of this 
paper. But the notation $D(B{\Gamma_m}_+)$ and the way 
classifying spaces related to $G$ appear in (\ref{colimit}) was one of the main 
motivations of the author to study the map $\Phi$. 
% end of motivation

Above and elsewhere in this paper, by ``$\holim$," we mean the 
homotopy limit given by 
\cite[Definition 18.1.8]{Hirschhorn}. 

This paragraph can be ignored by the reader who only wants the 
main ideas of this paper. As explained in Definition \ref{definition_phi} 
and Remark 
\ref{remark_collection} regarding ``the map $\Phi$," there is a map 
$\Phi = \Phi_{\{U\}}$ for 
every cofinal subcollection $\{U\}$ of $\{N\}$, and as explained 
in Remark \ref{remark_collection}, all of these maps are 
``the same one map, up to isomorphism" in a precise sense and 
if any one of them is a weak equivalence, then all of them are weak equivalences. We have chosen this terminology and notation to 
make the Introduction easier to read and for flexibility with proofs 
(including proofs of future results). 

Throughout this Introduction and everywhere else in this paper, if 
$Z$ is any spectrum, then by $\pi_t(Z)$, where $t \in \mathbb{Z}$, 
we mean the homotopy groups $[S^t, Z]$ of morphisms in the 
homotopy category of $Sp^\Sigma$, where here, $S^t$ denotes a 
fixed cofibrant and fibrant model for the $t$-th suspension of the sphere spectrum. 

\subsection*{Acknowledgements} I thank Thomas Credeur and Philip Hackney for helpful conversations. 

\section{The construction of $\Phi$, with target equivalent to $X^{hG}$}\label{map_phi} 
Let $G$ be any profinite group, let $X$ be a discrete $G$-spectrum with trivial $G$-action, and let the map
$\eta \: X \xrightarrow{\,\simeq\,} X_f$ be a functorial fibrant replacement in 
$Sp^\Sigma$. 
Thus, by functoriality, the map $\eta$ is $G$-equivariant and the $G$-action on $X_f$ is trivial. We regard $X_f$ as a discrete $G$-spectrum, and it follows that 
the map $\eta$
is a trivial cofibration in $\Sigma\mathrm{Sp}_G$. The commutative diagram
\[\xymatrix{
X \ar[r]^-{\theta_{\scriptscriptstyle{G}}}_-{\simeq} \ar[d]^-{\simeq}_-\eta & X_{fG} \ar[d] \\X_f \ar[r] & \ast}\] 
of discrete $G$-spectra yields a weak equivalence 
\[\lambda \: X_f \xrightarrow{\,\simeq\,} X_{fG}\] of discrete $G$-spectra. 

Let $\{U\}$ be a collection of open normal subgroups of $G$ that is 
cofinal in the collection $\{N\}$ of all open normal subgroups of $G$.
Since the $G$-action on $X_f$ is trivial, for each $U$ there is 
the $G/U$-equivariant map 
\[\lambda^U \: X_f = (X_f)^U \to (X_{fG})^U,\] where here, $X_f$ is regarded as a 
$G/U$-spectrum with the trivial action. 

To go further, 
we need to recall several constructions. Following 
\cite[Section 2.4]{joint}, given a finite set $K$ 
and a spectrum $Z$, we let $\mathrm{Map}(K,Z)$ be the spectrum such that for each $m, n \geq 0$, its $m$th pointed simplicial set $\mathrm{Map}(K,Z)_m$ has $n$-simplices equal to $\mathrm{Map}(K,(Z_m)_n)$, 
which is the set of functions from $K$ to $(Z_m)_n$, the set of $n$-simplices 
of $Z_m$. 
It is helpful to note that there is an isomorphism 
\[\mathrm{Map}(K,Z) \cong \prod_K Z.\] 
Also, as in 
\cite[Section 3.2]{joint}, when $K$ is a finite group, 
we let $\mathrm{Map}(K,-)$ be the coaugmented comonad on the category of spectra 
(this comonad is 
written as $\mathrm{Map}^c(K,-)$ in the more general discussion in 
\cite{joint}; we omit the superscript ``$c$" (which is for ``continuous") because here, $K$ is finite (and naturally has the discrete topology)) 
that has the following properties, given a $K$-spectrum $Y$:\begin{enumerate}
\item[$\bullet$]
there is a cosimplicial spectrum 
\[\mathrm{Map}(K^\bullet, Y) \: \Delta \to Sp^\Sigma, \ \ \ 
[n] \mapsto \mathrm{Map}(K^\bullet, Y)([n])\] (denoted by $\Gamma_K^\bullet Y$ in 
\cite[Section 3.2]{joint}); 
\item[$\bullet$] 
for each $n \geq 0$, there are natural isomorphisms of spectra
\[\mathrm{Map}(K^\bullet, Y)([n]) \cong \mathrm{Map}(K^n, Y) \cong \prod_{K^n} Y,\] where $K^n$ is the $n$-fold cartesian product 
of $K$ ($K^0$ is the trivial group $1$), by \cite[Section 4.6]{joint}; 
\item[$\bullet$]
by \cite[Theorem 3.2.1]{joint}, if $Y$ is fibrant as a spectrum, then 
there is a 
natural equivalence
\[Y^{hK} \simeq \holim_\Delta \mathrm{Map}(K^\bullet, Y);\] 
\item[$\bullet$] 
(here, $Y$ need not be fibrant) there is an isomorphism
\[\lim_\Delta \mathrm{Map}(K^\bullet, Y) \cong 
\mathrm{equa}\mspace{-1mu}\bigl[\xymatrix{\mspace{-4.5mu}Y 
\ar@<.9ex>[r]%^-{\scriptscriptstyle{d^0}} 
\ar@<-.5ex>[r]%_-{\scriptscriptstyle{d^1}} 
& \mathrm{Map}(K,Y)}
\mspace{-4.5mu}\bigr] = Y^K,\] where the middle expression 
above is the 
equalizer of the 
``first two" coface maps (which can be denoted as $d^0$ and $d^1$) 
of $\mathrm{Map}(K^\bullet, Y)$ (the above equality 
between the equalizer and $Y^K$ is the homotopical avatar of the fact that for each $t \in \mathbb{Z}$, the kernel of the group homomorphism $\pi_t(d^0)-\pi_t(d^1)$ is $(\pi_t(Y))^{K} = H^0(K,\pi_t(Y))$ (for example, see \cite[Theorem 3.2.1; proof of Proposition 3.5.3]{joint})). \end{enumerate}

Now we put the threads represented by the 
two preceding paragraphs together. For each $U$, the $G/U$-equivariant map $\lambda^U \: X_f \to (X_{fG})^U$ induces a map  
\[\holim_\Delta \mathrm{Map}((G/U)^{\bullet}, X_f) 
\to \holim_\Delta \mathrm{Map}((G/U)^{\bullet}, (X_{fG})^U)\] of spectra, 
which is natural as $U$ varies, and hence, we have the 
following.
\begin{Def}\label{definition_phi}
If $G$ is any profinite group, $X$ a discrete $G$-spectrum with trivial $G$-action, and $\{U\}$ any collection of open normal subgroups of $G$ that is cofinal in $\{N\}$, then there is the map 
\[\Phi \: \colim_U \holim_\Delta \mathrm{Map}((G/U)^{\bullet}, X_f) 
\to \colim_U 
\holim_\Delta \mathrm{Map}((G/U)^{\bullet}, (X_{fG})^U).\] 
The map $\Phi$ is the map $\Phi$ described in Section \ref{intro}.
\end{Def}

% when G is finite
\begin{Thm}\label{WhenFinite}
In Definition \ref{definition_phi}, if $G$ is finite, then $\Phi$ is a weak equivalence.
\end{Thm}
\begin{proof}
Let $G$ be finite: $G$ has the discrete topology and $1$ is an open normal subgroup of $G$. Since $\{U\}$ is cofinal in $\{N\}$, 
$1 \in \{U\}$, and thus, the indexing category $\{U\}$ has $1$ as the terminal object. It follows that, for example, the source of $\Phi$ 
satisfies the isomorphism 
\[\colim_U \holim_\Delta \mathrm{Map}((G/U)^\bullet, X_f) 
\cong \holim_\Delta \mathrm{Map}((G/1)^\bullet, X_f),\] so that 
$\Phi$ is a weak equivalence if the map 
\[\holim_\Delta \mathrm{Map}(G^\bullet, \lambda) \: \holim_\Delta \mathrm{Map}(G^\bullet, X_f) \to \holim_\Delta 
\mathrm{Map}(G^\bullet, X_{fG})\] is a weak equivalence. Since 
$\lambda$ is a weak equivalence in $\Sigma\mathrm{Sp}_G$, it 
is a weak equivalence in $Sp^\Sigma$. Then for each $n \geq 0$, there is the 
commutative diagram 
\[\xymatrix{
\mathrm{Map}(G^\bullet, X_f)([n]) \ar[r] \ar[d]_-\cong & \mathrm{Map}(G^\bullet, X_{fG})([n]) \ar[d]^-\cong \\
\prod_{G^n} X_f \ar[r]^-\simeq & \prod_{G^n} X_{fG}\,, 
}\] in which the bottom horizontal map is a weak equivalence (in 
$Sp^\Sigma$, as, for example, in \cite[proof of Lemma 2.4.1]{joint}). Additionally, 
the source and target of the bottom horizontal map are 
fibrant spectra ($X_{fG}$ is fibrant as a spectrum, by 
\cite[Corollary 5.3.3]{joint}), so that the top horizontal map is a 
weak equivalence between fibrant spectra, 
and hence, the map $\holim_\Delta \mathrm{Map}(G^\bullet, \lambda)$ 
is a weak equivalence, by \cite[Theorem 18.5.3]{Hirschhorn}. 
\end{proof}
% end of when G is finite

Now we explain the identifications (alluded to in Section \ref{intro}) 
that allow us to regard the map $\Phi$ in Definition \ref{definition_phi} 
as a map \[\colim_N X^{hG/N} \to \colim_N (X^{hN})^{hG/N}.\]

Since the collection $\{U\}$ is cofinal in $\{N\}$, there is an isomorphism 
\[\colim_U \holim_\Delta \mathrm{Map}((G/U)^{\bullet}, X_f) \cong \colim_N \holim_\Delta \mathrm{Map}((G/N)^{\bullet}, X_f).\]
For each $N$, $X = X^N$ and $X_f = (X_f)^N$ are $G/N$-spectra and 
the weak equivalence $\eta \: X \xrightarrow{\,\simeq\,} X_f$ is $G/N$-equivariant, and hence, there is a weak equivalence
\[X^{hG/N} \xrightarrow{\,\simeq\,} (X_f)^{hG/N}.\] The above 
isomorphism and this last weak equivalence, 
together with the equivalence
\[\holim_\Delta \mathrm{Map}((G/N)^{\bullet}, X_f) \simeq (X_f)^{hG/N},\] imply that the source of $\Phi$ can be identified with 
$\colim_N X^{hG/N}$, as desired. 

For each open normal subgroup $N$ of $G$, $(X_{fG})^N$ is a 
fibrant spectrum (see \cite[proof of Corollary 5.3.3]{joint}), so that 
the isomorphism 
\[\colim_U \holim_\Delta \mathrm{Map}((G/U)^{\bullet}, (X_{fG})^U) \cong \colim_N \holim_\Delta \mathrm{Map}((G/N)^{\bullet}, (X_{fG})^N)\] implies that the target of $\Phi$ can be 
regarded as $\colim_N ((X_{fG})^N)^{hG/N}$. Now let $N$ be fixed. Since the 
fibrant replacement $\theta_{\scriptscriptstyle{G}} \: X \xrightarrow{\,\simeq\,} X_{fG}$ is a trivial cofibration of discrete $N$-spectra, 
it follows 
from the commutative diagram
\[\xymatrix{
X \ar[r]^-{\theta_{\scriptscriptstyle{N}}}_-{\simeq} 
\ar[d]^-{\simeq}_-{\theta_{\scriptscriptstyle{G}}} & X_{fN} \ar[d] \\X_{fG} \ar[r] & \ast}\] 
of discrete $N$-spectra that there is a weak equivalence 
\[\lambda_{\scriptscriptstyle{N}} \: X_{fG} \xrightarrow{\,\simeq\,} X_{fN}\] of discrete $N$-spectra. 
Furthermore, since $X_{fG}$ is fibrant as a discrete $N$-spectrum (by \cite[Proposition 3.3.1, (2)]{joint}) and 
\[(-)^N \: \Sigma{\mathrm{Sp}}_N \to Sp^\Sigma\] is a 
right Quillen functor \cite[Lemma 3.1.1]{joint}, the induced map
\[
(\lambda_{\scriptscriptstyle{N}})^N \: (X_{fG})^N \xrightarrow{\,\simeq\,} (X_{fN})^N = X^{hN}
\] is a 
weak equivalence (this conclusion is not original: for example, see \cite[proof of Proposition 3.3.1, (3)]{joint}). It follows that 
$(X_{fG})^N$ can be identified with $X^{hN}$, and thus, the target of 
$\Phi$ can be identified with 
$\colim_N (X^{hN})^{hG/N}$. This completes the construction of the map $\Phi$, as described in Section \ref{intro}. 

\begin{Rk}\label{remark_collection}
The definition of $\Phi$ depends on the cofinal collection 
$\{U\}$, and so, when necessary, we can refer to $\Phi$ as 
$\Phi_{\{U\}}$. If $\{V\}$ is any other collection of open normal subgroups of $G$ that is cofinal in $\{N\}$, then there is the commutative 
diagram 
\[\xymatrix{
\colim_U \holim_\Delta \mathrm{Map}((G/U)^{\bullet}, X_f) 
\ar[r]^-{\Phi_{\{U\}}} \ar[d]_-\cong & \colim_U 
\holim_\Delta \mathrm{Map}((G/U)^{\bullet}, (X_{fG})^U) \ar[d]^-\cong\\
\colim_V \holim_\Delta \mathrm{Map}((G/V)^{\bullet}, X_f) 
\ar[r]^-{\Phi_{\{V\}}} & \colim_V 
\holim_\Delta \mathrm{Map}((G/V)^{\bullet}, (X_{fG})^V),}\]
in which the sources of $\Phi_{\{U\}}$ and $\Phi_{\{V\}}$ 
are isomorphic, and similarly their targets, so that in the sense made 
precise by this diagram, $\Phi_{\{U\}}$ is ``unique up to isomorphism." 
For this reason, we do not refrain (in Section \ref{intro_proper}, this section, and elsewhere) from using language that might make it seem like $\Phi$ is 
unique or independent of the choice of $\{U\}$. Of course, the 
above diagram shows that $\Phi_{\{U\}}$ is a weak equivalence if and only 
if $\Phi_{\{V\}}$ is a weak equivalence.
\end{Rk}

Now we construct the weak equivalence
\[\Psi \: X^{hG} \xrightarrow{\,\simeq\,} \colim_N (X^{hN})^{hG/N}.\] 
For each $U$, there is the canonical map 
\[\iota_{\scriptscriptstyle{U}} \: \lim_\Delta \mathrm{Map}((G/U)^{\bullet}, (X_{fG})^U) \to 
\holim_\Delta \mathrm{Map}((G/U)^{\bullet}, (X_{fG})^U)\] 
between the limit and homotopy limit, and the source of this map 
satisfies the isomorphism 
\[\lim_\Delta \mathrm{Map}((G/U)^{\bullet}, (X_{fG})^U) 
\cong ((X_{fG})^U)^{G/U} = (X_{fG})^G = X^{hG},\] which yields 
the map
\[\widehat{\iota}_{\scriptscriptstyle{U}} \: X^{hG} \xrightarrow{\,\cong\,} 
\lim_\Delta \mathrm{Map}((G/U)^{\bullet}, (X_{fG})^U) 
\xrightarrow{\, {\textstyle{\iota}}_{\scriptscriptstyle{U}}\, }
\holim_\Delta \mathrm{Map}((G/U)^{\bullet}, (X_{fG})^U).\] Then we define 
\[\Psi := \colim_U \widehat{\iota}_{\scriptscriptstyle{U}} \: X^{hG} = 
\colim_U X^{hG} \longrightarrow 
\colim_U \holim_\Delta \mathrm{Map}((G/U)^{\bullet}, (X_{fG})^U),\] 
where, before any identifications are made, the target of $\Psi$ is 
exactly equal to the target of $\Phi$ (that is, $\Phi_{\{U\}}$), so that as in the above discussion about the target of 
$\Phi$, the target of $\Psi$ can be written as $\colim_N (X^{hN})^{hG/N}$. 

\begin{Rk}\label{remark_psi}
All the comments in Remark \ref{remark_collection} go through for $\Psi$, mutatis mutandis. Also, in the context of $\Psi$, the last 
sentence of Remark \ref{remark_collection} can be strengthened to say that every morphism $\Psi_{\{U\}}$ is a weak equivalence (as we show below).  
\end{Rk}

To show that $\Psi$, a filtered colimit, is a weak equivalence, it 
suffices to show that each $\widehat{\iota}_{\scriptscriptstyle{U}}$ is a weak equivalence between fibrant spectra. Since $(-)^G \: \Sigma\mathrm{Sp}_G \to Sp^\Sigma$ is a right Quillen functor, $X^{hG}$ is a fibrant spectrum. Also, for each $n \geq 0$, the isomorphism
\[\mathrm{Map}((G/U)^{n}, (X_{fG})^U) \cong \prod_{(G/U)^{n}} 
(X_{fG})^U\] implies that the spectrum of $n$-cosimplices of 
$\mathrm{Map}((G/U)^{\bullet}, (X_{fG})^U)$ is fibrant (since $(X_{fG})^U$ is fibrant), and 
hence, $\holim_\Delta \mathrm{Map}((G/U)^{\bullet}, (X_{fG})^U)$ is 
fibrant, by \cite[Theorem 18.5.2]{Hirschhorn}. Now we only need 
to show that 
each ${\iota}_{\scriptscriptstyle{U}}$ is a weak equivalence. 
 
For each $U$, there is the equivalence
\[\holim_\Delta \mathrm{Map}((G/U)^{\bullet}, (X_{fG})^U) 
\simeq ((X_{fG})^U)^{hG/U},\] and since $(X_{fG})^U$ is a fibrant discrete $G/U$-spectrum \cite[Proposition 3.3.1, (1)]{joint}, the functorial fibrant replacement $\theta_{\scriptscriptstyle{G/U}} \: (X_{fG})^U \xrightarrow{\,\simeq\,} ((X_{fG})^U)_{fG/U}$ in $\Sigma\mathrm{Sp}_{G/U}$ induces the weak equivalence 
%\begin{equation}\label{fibrantG}\zig
\[
(\theta_{\scriptscriptstyle{G/U}})^{G/U} \: 
X^{hG} = ((X_{fG})^U)^{G/U} \xrightarrow{\,\simeq\,} (((X_{fG})^U)_{fG/U})^{G/U} = ((X_{fG})^U)^{hG/U}
\]
%\end{equation} 
(this argument is a special case of \cite[Proposition 3.5.1]{joint}). These observations yield the equivalence 
\[\holim_\Delta \mathrm{Map}((G/U)^{\bullet}, (X_{fG})^U) 
\simeq X^{hG},\] which naturally fits into the ``bigger 
picture"
\[X^{hG} \xrightarrow{\,\cong\,} \lim_\Delta \mathrm{Map}((G/U)^{\bullet}, (X_{fG})^U) \xrightarrow{\,\textstyle{\iota}_{\scriptscriptstyle{U}}\,} \holim_\Delta \mathrm{Map}((G/U)^{\bullet}, (X_{fG})^U) \simeq X^{hG}.\] It follows from this picture that 
the canonical map $\iota_{\scriptscriptstyle{U}}$ is a weak equivalence, as desired. 
% the canonical map $\iota_{\scriptscriptstyle{U}}$ is an explicit model 
% for the weak equivalence $(\theta_{\scriptscriptstyle{G/U}})^{G/U}$. 
This 
completes our proof that $\Psi$ is a weak equivalence. 

\section{The proof of Theorem \ref{vanishing}}\label{section_vanishing}

As before, $G$ is a profinite group and $X$ is a discrete $G$-spectrum with trivial $G$-action. Additionally, we suppose that $\{U\}$ is a cofinal subcollection of $\{N\}$, such that for each $U$ and every integer $t$, 
$H^s_c(U, \pi_t(X)) = 0$, whenever $s>0$. To prove part (a) of 
Theorem 
\ref{vanishing}, we want to obtain that 
$\Phi = \Phi_{\{V\}}$ is a weak equivalence, where $\{V\}$ is any 
cofinal subcollection of $\{N\}$. Thus, to prove Theorem \ref{vanishing}, 
we only need to prove that (i) $\Phi_{\{U\}}$ is a weak equivalence (see Remark \ref{remark_collection}), and, for each $U$, there is (ii) 
an equivalence
$X^{hG/U} \simeq X^{hG}$ and (iii) a weak equivalence $X 
\xrightarrow{\,\simeq\,} X^{hU}$. To this end, recall from Definition 
\ref{definition_phi} that 
\[\Phi_{\{U\}} = \colim_U 
\holim_\Delta \mathrm{Map}((G/U)^\bullet, \lambda^U),\] where 
each map \[h_{\scriptscriptstyle{U}} := 
\holim_\Delta \mathrm{Map}((G/U)^\bullet, \lambda^U)\] is the 
``lengthy centerpiece" in the ``picture"
\[X^{hG/U} \simeq 
\holim_\Delta \mathrm{Map}((G/{U})^{\bullet}, X_f) 
\xrightarrow{\,\scriptstyle{h}_{\scriptscriptstyle{U}}\,} \holim_\Delta \mathrm{Map}((G/{U})^{\bullet}, (X_{fG})^U) 
\simeq X^{hG}\] (the two equivalences on the ends are explained in Section \ref{map_phi}). Also, in Section \ref{map_phi} we saw that the target of $h_{\scriptscriptstyle{U}}$ is a fibrant spectrum, and similarly, so is the source of this map. Therefore, to show (i) and (ii), we only need to verify 
that for each $U$, the map $h_{\scriptscriptstyle{U}}$ 
is a weak equivalence. The proof of (iii) will be a byproduct of our argument for (i) and (ii). 

It will be helpful to recall that if $K$ is a finite group and $Y$ is a $K$-spectrum that is also fibrant as a spectrum, then associated to the equivalence $Y^{hK} \simeq \holim_\Delta \mathrm{Map}(K^\bullet, Y)$ is the conditionally convergent 
homotopy spectral sequence 
\[E_2^{s,t} \Longrightarrow \pi_{t-s}(\holim_\Delta \mathrm{Map}(K^\bullet, Y)) \cong \pi_{t-s}(Y^{hK}),\] where
\[E_2^{s,t} = H^s\bigl[\pi_t(\mathrm{Map}(K^\ast, Y))\bigr] 
% \cong 
% H^s\bigl[\mathrm{Map}(K^\ast, \pi_t(Y))\bigr] 
\cong H^s(K, \pi_t(Y)),\] 
with $\pi_t(\mathrm{Map}(K^\ast, Y))$ denoting the usual cochain complex associated to the cosimplicial abelian group 
$\pi_t(\mathrm{Map}(K^\bullet, Y))$ (by, for example, 
\cite[proof of Proposition 3.5.3]{joint}; this homotopy 
spectral sequence gives one way of obtaining 
the homotopy fixed point spectral sequence for $Y^{hK}$).  

Now fix any $U' \in \{U\}$. The map $\lambda^{U'} \: X_f \to (X_{fG})^{U'}$ induces 
the morphism 
\[\mathrm{Map}((G/{U'})^{\bullet}, X_f) 
\to \mathrm{Map}((G/{U'})^{\bullet}, (X_{fG})^{U'})\] of cosimplicial 
spectra, which gives 
the morphism 
\[\xymatrix{
H^s(G/{U'}, \pi_t(X)) \ar[d] \ar@{=>}[r] & \pi_{t-s}(\holim_\Delta \mathrm{Map}((G/{U'})^{\bullet}, X_f)) \ar[d]^-{\pi_{t-s}(h_{\scriptscriptstyle{U'}})}\\
H^s(G/{U'}, \pi_t((X_{fG})^{{U'}})) \ar@{=>}[r] & \pi_{t-s}(\holim_\Delta \mathrm{Map}((G/{U'})^{\bullet}, (X_{fG})^{U'}))}\] of 
conditionally convergent homotopy spectral sequences. To show 
that $h_{\scriptscriptstyle{U'}}$ is a weak equivalence, it suffices 
to show that for each $t \in \mathbb{Z}$, the $G/{U'}$-equivariant 
composition 
\[\pi_t(X) \xrightarrow{\,\cong\,} \pi_t(X_f) \xrightarrow{\,\pi_t(\lambda^{U'})\,} \pi_t((X_{fG})^{U'})\] is an 
isomorphism, for this implies that the map 
between the 
$E_2$-terms of the above two spectral sequences is an isomorphism, which then gives that the map between the two abutments is an 
isomorphism, yielding the desired conclusion.

For every $U$, $U \cap U'$ is an open normal subgroup of $G$, so 
that by the cofinality of $\{U\}$ in $\{N\}$, we can choose one 
$U'' \in \{U\}$ such that $U \cap U'$ contains $U''$. Then the 
collection $\{U\}$ induces a collection $\{U''\}$, which is cofinal in the collection of all open normal subgroups of $U'$ (each $U'' = U'' \cap U'$ is 
open in $U'$; since $\{U\}$ is a directed poset and $\{U''\}$ is a subset of $\{U\}$, it follows easily that 
$\{U''\}$ is a directed poset; if $V$ is an open normal subgroup of 
$U'$, then $V$ is an open subgroup of $G$, and hence, $V$ contains a subgroup $V'$ that is open and normal in $G$, so that by the cofinality of $\{U\}$ in $\{N\}$, $V'$ contains some $W \in \{U\}$: by definition, there is 
some $W'' \in \{U''\}$ such that $W \cap U'$ contains $W''$, so that  
$W'' \subset V' \cap U' \subset V \cap U' = V$, giving $W'' \subset V$, as 
desired). Since $\{U''\}$ is a subset of $\{U\}$, 
$H^s_c(U'', \pi_t(X)) = 0$, for 
all $s > 0$, every $U''$, and each integer $t$, and therefore, the conditionally convergent 
homotopy fixed point spectral sequence 
\begin{equation}\label{hfps}\zig
H^s_c(U', \pi_t(X)) \Longrightarrow \pi_{t-s}(X^{hU'})\end{equation} exists (see 
\cite[page 911: condition (ii), (1.2)]{nyjm2nd}, \cite[proof of Theorem 3.2.1]{joint}, and \cite[proof of Theorem 7.4]{cts}). 
By hypothesis, this spectral sequence collapses at the $E_2$-page, so that for every 
integer $t$,
\begin{equation}\label{isomorphisms}\zig\pi_t(X^{hU'}) \cong H^0_c(U', \pi_t(X)) \cong (\pi_t(X))^{U'} = \pi_t(X).\end{equation}  

To complete the proof of Theorem \ref{vanishing}, we need a recollection from \cite{joint} so that we can partly unpack spectral sequence (\ref{hfps}). 
Let $H$ be any profinite group. As in \cite[Section 2.4]{joint}, given 
a spectrum $Z$, the spectrum $\mathrm{Map}^c(H,Z)$ of 
continuous maps is defined by
\[\mathrm{Map}^c(H,Z) = \colim_{N_H} \mathrm{Map}(H/{N_H}, Z),\] 
where the colimit is over all the open normal subgroups $N_H$ of $H$. 
Also, by \cite[Section 3.2]{joint}, 
%if $H$ is any profinite group and 
%
if $Y$ is a discrete 
$H$-spectrum, then there is a 
cosimplicial spectrum $\mathrm{Map}^c(H^\bullet, Y)$ -- written as 
$\Gamma^\bullet_H Y$ in \cite{joint} -- with the property that for each $[n] \in \Delta$, there is a natural isomorphism 
\[\mathrm{Map}^c(H^\bullet, Y)([n]) \cong \mathrm{Map}^c(H^n, Y),\] 
by \cite[Section 4.6]{joint}, and as in Section \ref{map_phi},
\[
\lim_\Delta \mathrm{Map}^c(H^{\bullet}, Y) 
\cong \mathrm{equa}\mspace{-1mu}\bigl[\mspace{-4.5mu}\xymatrix{Y 
\ar@<.9ex>[r]%^-{\scriptscriptstyle{d^0}} 
\ar@<-.5ex>[r]%_-{\scriptscriptstyle{d^1}} 
& \mathrm{Map}^c(H, Y)}\mspace{-4.5mu}\bigr] 
= Y^H.\] 
Then spectral sequence (\ref{hfps}) is just the homotopy spectral 
sequence associated to the left-hand side in the equivalence
\[\holim_\Delta \mathrm{Map}^c((U')^\bullet, X_{fG}) \simeq X^{hU'},\] 
which follows from the properties of the collection 
$\{U''\}$ (see the citations that were given above for this), and the 
isomorphism $\pi_t(X^{hU'}) \cong \pi_t(X)$ of (\ref{isomorphisms}) 
comes from the above equivalence and an isomorphism 
\begin{equation}\label{last_one}\zig
\pi_t(\holim_\Delta \mathrm{Map}^c((U')^\bullet, X_{fG})) 
\xleftarrow{\,\cong\,} \pi_t(X),\end{equation} valid for all integers $t$ 
(given by the collapsing of spectral sequence (\ref{hfps}), as before), that is induced by the 
composition
\[X \xrightarrow[\simeq]{\eta} X_f \xrightarrow{\lambda^{U'}} 
(X_{fG})^{U'} \xrightarrow{\cong} \lim_\Delta \mathrm{Map}^c((U')^\bullet, X_{fG}) \xrightarrow{\iota(U')} \holim_\Delta 
\mathrm{Map}^c((U')^\bullet, X_{fG}),\] where $\iota(U')$ is the 
canonical map between the limit and homotopy limit. The isomorphism in (\ref{last_one}) implies that this composition is a weak equivalence, 
and since there are equivalences
\[(X_{fG})^{U'} \xrightarrow[\,\simeq\,]{(\lambda_{\scriptscriptstyle{U'}})^{U'}} X^{hU'} \simeq \holim_\Delta \mathrm{Map}^c((U')^\bullet, X_{fG})\] (see Section \ref{map_phi} for 
weak equivalence $(\lambda_{\scriptscriptstyle{U'}})^{U'}$), it follows that $\iota(U')$ is a weak equivalence, 
and hence, $\lambda^{U'}$ is a weak equivalence. Therefore, 
the map $\pi_t(\lambda^{U'})$ is an isomorphism for every $t \in 
\mathbb{Z}$, completing the proof that $h_{\scriptscriptstyle{U'}}$ is a 
weak equivalence. 

Finally, since $\lambda^{U'}$ is a weak equivalence, the 
composition
\[X \xrightarrow[\,\simeq\,]{\eta} X_f \xrightarrow[\,\simeq\,]{\lambda^{U'}} (X_{fG})^{U'} \xrightarrow[\,\simeq\,]{(\lambda_{\scriptscriptstyle{U'}})^{U'}} X^{hU'}\] defines the 
desired weak equivalence $X \xrightarrow{\,\simeq\,} X^{hU'},$ 
%below is correct for Section 3
which completes the proof of Theorem \ref{vanishing}. 

\section{Discrete $G$-spectra with trivial action that are bounded above}\label{section_4}

In this section, we use the notation of Section \ref{map_phi}. Thus, $G$ 
is an arbitrary profinite group, 
$\{U\}$ is any fixed cofinal subcollection of $\{N\}$, and, as usual,  
$\Phi$ denotes $\Phi_{\{U\}}$. After the preliminary result below, 
which does not require $X$ to be bounded above, 
we give a proof of Theorem \ref{coconn}. 

In the following result, 
the construction $\mathrm{Map}^c(G^\bullet, Y)$ for a discrete 
$G$-spectrum $Y$ -- recalled in Section \ref{section_vanishing} -- 
is from \cite[Section 3.2]{joint}.

\begin{Thm}\label{diagram}
If $G$ is any profinite group, $X$ is a discrete $G$-spectrum whose $G$-action is trivial, and $\{U\}$ is a collection of open normal subgroups that is cofinal in the collection $\{N\}$ of all open normal subgroups of $G$, then there is the commutative diagram 
\[\xymatrix{
\colim_U \holim_\Delta \mathrm{Map}((G/U)^{\bullet}, X_f) 
\ar[r]^-\Phi \ar[d]_-{s_1} & \colim_U 
\holim_\Delta \mathrm{Map}((G/U)^{\bullet}, (X_{fG})^U) \ar[d]^-{s_2}\\
\holim_\Delta \colim_U \mathrm{Map}((G/U)^{\bullet}, X_f) 
\ar[r]^-\simeq \ar[d]_-\cong & \holim_\Delta 
\colim_U \mathrm{Map}((G/U)^{\bullet}, (X_{fG})^U) \ar[d]^-\cong\\ 
\holim_\Delta \mathrm{Map}^c(G^{\bullet}, X_f) 
\ar[r]^-\simeq & \holim_\Delta \mathrm{Map}^c(G^{\bullet}, X_{fG}),}
\] where $s_1$ and $s_2$ are the canonical colim/holim interchanges, 
the lower two vertical maps are isomorphisms induced by the colimits, and the middle and bottom horizontal maps are weak equivalences induced by the collection $\{\lambda^U \: X_f \to (X_{fG})^U\}_{\{U\}}$ and $\lambda \: X_f \xrightarrow{\,\simeq\,} X_{fG}$, respectively.
\end{Thm}

\begin{proof} 
To save space, we sometimes write $\mathcal{M}$ 
%and $\mathcal{M}^c$ 
for ``$\mathrm{Map}$." 
%and ``$\mathrm{Map}^c$", respectively. 
Since the collection $\{U\}$ is cofinal in $\{N\}$, for each $[n] \in \Delta$, there are isomorphisms
\[\colim_U (\mathrm{Map}((G/U)^\bullet, X_f)([n])) \xrightarrow{\cong} \colim_U \mathcal{M}((G/U)^{n}, X_f) \xrightarrow{\cong} \colim_N \mathcal{M}((G/N)^{n}, X_f)\] and 
\[\colim_N \mathcal{M}((G/N)^{n}, X_f) \xrightarrow{\,\cong\,} \mathrm{Map}^c(G^n, X_f) \xrightarrow{\,\cong\,} 
\mathrm{Map}^c(G^\bullet, X_f)([n])
\] 
(for the next-to-last isomorphism, see, for example, \cite[proof of Lemma 6.5.4, (b)]{Ribes}),
whose composition is the isomorphism 
\[\colim_U (\mathrm{Map}((G/U)^\bullet, X_f)([n])) 
\xrightarrow{\,\cong\,} 
\mathrm{Map}^c(G^\bullet, X_f)([n]),
\]
% it was the next line
%\[\colim_U \mathrm{Map}((G/U)^{n}, X_f) \xrightarrow{\,\cong\,} \colim_N \mathrm{Map}((G/N)^{n}, X_f) \xrightarrow{\,\cong\,} \mathrm{Map}^c(G^n, X_f)\]  
and thus, the 
lower left vertical map is an isomorphism (as in \cite[proof of Theorem 8.2.5: page 5038]{joint}). Similarly, for each $[n]$, there are isomorphisms
\[\colim_U \mathrm{Map}((G/U)^{n}, (X_{fG})^U) \xrightarrow{\,\cong\,} 
\mathrm{Map}^c(G^{n}, \colim_N (X_{fG})^N) \xrightarrow{\,\cong\,} 
\mathrm{Map}^c(G^{n}, X_{fG}),\] where the second one follows 
from the isomorphism 
$\colim_N (X_{fG})^N \xrightarrow{\,\cong\,} X_{fG}$ in $\Sigma\mathrm{Sp}_G$, which exists 
because $X_{fG}$ is a discrete $G$-spectrum, so that 
the lower right vertical map is also an isomorphism. 

It will be helpful to note that if $Z$ is a fibrant spectrum 
and $n \geq 1$, then the 
isomorphism 
\[\mathrm{Map}^c(G^n, Z) \cong \colim_V \prod_{G^n/V} Z,\] where 
the colimit is over all the open normal subgroups $V$ of $G^n$, implies 
that $\mathrm{Map}^c(G^n, Z)$ is a fibrant spectrum. Now the 
argument follows, with a few changes, the proof of Theorem 
\ref{WhenFinite}. By \cite[Lemma 2.4.1]{joint}, the maps 
\[\mathrm{Map}^c(G^{n}, X_f) \to \mathrm{Map}^c(G^{n}, X_{fG}), \ \ \ n \geq 0,\] are weak equivalences between fibrant spectra: this fact, 
coupled with the commutative diagrams
\[\xymatrix{
\mathrm{Map}^c(G^\bullet, X_f)([n]) \ar_-\cong[d] \ar[r] & \mathrm{Map}^c(G^\bullet, X_{fG})([n]) \ar^-\cong[d] \\
\mathrm{Map}^c(G^n, X_f) \ar^-\simeq[r] & \mathrm{Map}^c(G^n, X_{fG})}
\] for each $n \geq 0$, implies that the upper horizontal 
map in each of these diagrams is a weak equivalence between 
fibrant spectra, and hence, the map $\holim_\Delta \mathrm{Map}^c(G^\bullet, \lambda)$ (the bottom horizontal map in the statement of Theorem \ref{diagram}) is a weak equivalence. Therefore, 
the map 
$\holim_\Delta \colim_U \mathrm{Map}((G/U)^\bullet, \lambda^U)$ 
(the middle horizontal map in the statement of Theorem \ref{diagram}) 
is a weak equivalence. 
%(the bottom horizontal map in the diagram in the statement of Theorem \ref{diagram}) is a weak equivalence, 
%middle horizontal map in the diagram in the statement of Theorem 
%\ref{diagram} 
\end{proof}

% already verified: at this point in the text, it is clear from the preceding text that G is an arbitrary profinite group

Now we give the proof of Theorem 
\ref{coconn}: as above, we let 
$X$ be a discrete $G$-spectrum with trivial $G$-action; 
additionally, we assume that $X$ is bounded above, so that 
there exists some $t_0 \in \mathbb{Z}$ such that $\pi_t(X) = 0$, 
whenever $t > t_0$; and we note that we only need 
to prove that $\Phi$ is a weak equivalence. 
The argument begins by considering the commutative diagram in Theorem \ref{diagram} further. To show that $\Phi$ is a 
weak equivalence, it suffices to show that $s_1$ and $s_2$ are 
weak equivalences. We first show that 
\[s_1 \: \colim_U \holim_\Delta \mathrm{Map}((G/U)^\bullet, X_f) \to 
\holim_\Delta \colim_U \mathrm{Map}((G/U)^\bullet, X_f)\]
is a weak equivalence. 
For every $U$, all $n \geq 0$, and each integer $t > t_0$, there are isomorphisms
\[
\pi_t(\mathrm{Map}((G/U)^\bullet, X_f)([n])) \cong \pi_t(\mathrm{Map}((G/U)^n, X_f)) \cong \prod_{\scriptscriptstyle{(G/U)^n}} \pi_t(X) = 0,\] 
%\\ & \cong \prod_{\scriptscriptstyle{(G/U)^n}} \pi_t(X) = 0,
%\end{align*} 
and thus, \cite[Proposition 3.4]{Mitchell} implies that 
$s_1$ is a weak equivalence.  

Now we consider $s_2$, by placing it within the 
commutative diagram
\[\xymatrix{
\displaystyle{\colim_U \lim_\Delta} 
\, \mathrm{Map}((G/U)^{\bullet}, (X_{fG})^U) \ar[d]_-{s_4} 
\ar[r]^-{s_3}  
& \displaystyle{\colim_U \holim_\Delta} 
\, \mathrm{Map}((G/U)^{\bullet}, (X_{fG})^U) \ar[d]^-{s_2}
\\
\displaystyle{\lim_\Delta \colim_U}\, 
\mathrm{Map}((G/U)^{\bullet}, (X_{fG})^U) \ar[r]^-{s_5} 
& \displaystyle{\holim_\Delta \colim_U}\, 
\mathrm{Map}((G/U)^{\bullet}, (X_{fG})^U),}
\] 
where $s_3 := \colim_U 
%{\scriptscriptstyle{U}} 
\iota_{\scriptscriptstyle{U}}$ is immediately seen 
to be a weak equivalence (see the definition of weak equivalence 
$\Psi$ in Section 
\ref{map_phi}), $s_4$ is the canonical map that interchanges the colimit and limit, and $s_5$ is the map from a limit to a homotopy limit. Since $s_4$ can be viewed as interchanging a filtered colimit and an equalizer, which is a finite limit, $s_4$ is an isomorphism. 

To complete the proof that $s_2$ is a weak equivalence, we only 
need to show that $s_5$ is a weak equivalence. 
The map $s_5$ and the lower right vertical isomorphism 
in the diagram in Theorem \ref{diagram} fit into the commutative 
diagram
\[\xymatrix{
\displaystyle{\lim_\Delta \colim_U}\, 
\mathrm{Map}((G/U)^{\bullet}, (X_{fG})^U) \ar[d]_-\cong \ar[r]^-{s_5} 
& \displaystyle{\holim_\Delta \colim_U}\, 
\mathrm{Map}((G/U)^{\bullet}, (X_{fG})^U) \ar[d]^-\cong\\
\displaystyle{\lim_\Delta}\, 
\mathrm{Map}^c(G^{\bullet}, X_{fG}) \ar[r]^-{s_6} & 
\displaystyle{\holim_\Delta}\, 
\mathrm{Map}^c(G^{\bullet}, X_{fG}),}\] in which the left 
vertical map is an isomorphism for the same reason that the right 
vertical map is one, and thus, we only 
need to show that the bottom horizontal map $s_6$ is a weak 
equivalence, which is a consequence of the following two observations:
%\begin{enumerate}
since $X$ is bounded above, there is an equivalence
%\begin{equation}\label{Uinsource}\zig
\[
X^{hG} \simeq \holim_\Delta \mathrm{Map}^c(G^{\bullet}, X_{fG}),
\] by \cite[page 911, especially condition (iii)]{nyjm2nd} and 
\cite[proof of Theorem 3.2.1]{joint}; and
%\item[$\bullet$] 
as in Sections \ref{map_phi} and \ref{section_vanishing}, 
%\begin{align*}
\[
\lim_\Delta \mathrm{Map}^c(G^{\bullet}, X_{fG}) 
\cong \mathrm{equa}\mspace{-1mu}\bigl[\mspace{-4.5mu}\xymatrix{X_{fG} 
\ar@<.9ex>[r]%^-{\scriptscriptstyle{d^0}} 
\ar@<-.5ex>[r]%_-{\scriptscriptstyle{d^1}} 
& \mathrm{Map}^c(G, X_{fG})}\mspace{-4.5mu}\bigr] 
= (X_{fG})^G = X^{hG}.\] 
%\end{align*} 
This completes the proof of Theorem \ref{coconn}.

\bibliographystyle{amsplain}

\begin{thebibliography}{999999}

\bibitem{BarthelBeaudry}
Tobias Barthel and Agn\`{e}s Beaudry.
\newblock Chromatic structures in stable homotopy theory. {A}vailable at
  ar{X}iv:1901.09004v2 [math.{AT}], 2019.

\bibitem{joint}
Mark Behrens and Daniel~G. Davis.
\newblock The homotopy fixed point spectra of profinite {G}alois extensions.
\newblock {\em Trans. Amer. Math. Soc.}, 362(9):4983--5042, 2010.

\bibitem{BuijsEtAl}
Urtzi Buijs, Yves F\'{e}lix, and Aniceto Murillo.
\newblock Rational homotopy of the (homotopy) fixed point sets of circle
  actions.
\newblock {\em Adv. Math.}, 222(1):151--171, 2009.

\bibitem{cts}
Daniel~G. Davis.
\newblock Homotopy fixed points for {$L_{K(n)}(E_n \wedge X)$} using the
  continuous action.
\newblock {\em J. Pure Appl. Algebra}, 206(3):322--354, 2006.

\bibitem{nyjm2nd}
Daniel~G. Davis.
\newblock Homotopy fixed points for profinite groups emulate homotopy fixed
  points for discrete groups.
\newblock {\em New York J. Math.}, 19:909--924, 2013.

\bibitem{DH}
Ethan~S. Devinatz and Michael~J. Hopkins.
\newblock Homotopy fixed point spectra for closed subgroups of the {M}orava
  stabilizer groups.
\newblock {\em Topology}, 43(1):1--47, 2004.

\bibitem{DwyerEtAl}
W.~G. Dwyer, J.~P.~C. Greenlees, and S.~Iyengar.
\newblock Duality in algebra and topology.
\newblock {\em Adv. Math.}, 200(2):357--402, 2006.

\bibitem{DwyerWilkerson}
W.~G. Dwyer and C.~W. Wilkerson.
\newblock Homotopy fixed-point methods for {L}ie groups and finite loop spaces.
\newblock {\em Ann. of Math. (2)}, 139(2):395--442, 1994.

\bibitem{hGal}
Paul~G. Goerss.
\newblock Homotopy fixed points for {G}alois groups.
\newblock In {\em The \v Cech centennial (Boston, MA, 1993)}, pages 187--224.
  Amer. Math. Soc., Providence, RI, 1995.

\bibitem{Goyo}
John~Octavius Goyo.
\newblock {\em The {S}ullivan model of the homotopy-fixed-point set}.
\newblock ProQuest LLC, Ann Arbor, MI, 1989.
\newblock Thesis (Ph.D.)--University of Toronto (Canada).

\bibitem{Hirschhorn}
Philip~S. Hirschhorn.
\newblock {\em Model categories and their localizations}, volume~99 of {\em
  Mathematical Surveys and Monographs}.
\newblock American Mathematical Society, Providence, RI, 2003.

\bibitem{SanFeliu}
Michael~J. Hopkins, Nicholas~J. Kuhn, and Douglas~C. Ravenel.
\newblock Morava {$K$}-theories of classifying spaces and generalized
  characters for finite groups.
\newblock In {\em Algebraic topology ({S}an {F}eliu de {G}u\'{\i}xols, 1990)},
  volume 1509 of {\em Lecture Notes in Math.}, pages 186--209. Springer,
  Berlin, 1992.

\bibitem{HSS}
Mark Hovey, Brooke Shipley, and Jeff Smith.
\newblock Symmetric spectra.
\newblock {\em J. Amer. Math. Soc.}, 13(1):149--208, 2000.

\bibitem{KuhnTransactions}
Nicholas~J. Kuhn.
\newblock The {M}orava {$K$}-theories of some classifying spaces.
\newblock {\em Trans. Amer. Math. Soc.}, 304(1):193--205, 1987.

\bibitem{LawsonHinfinity}
Tyler Lawson.
\newblock A note on {$H_\infty$} structures.
\newblock {\em Proc. Amer. Math. Soc.}, 143(7):3177--3181, 2015.

\bibitem{Mitchell}
Stephen~A. Mitchell.
\newblock Hypercohomology spectra and {T}homason's descent theorem.
\newblock In {\em Algebraic $K$-theory (Toronto, ON, 1996)}, volume~16 of {\em
  Fields Inst. Commun.}, pages 221--277. Amer. Math. Soc., Providence, RI,
  1997.

\bibitem{RavenelFiniteGroups}
Douglas~C. Ravenel.
\newblock Morava {$K$}-theories and finite groups.
\newblock In {\em Symposium on {A}lgebraic {T}opology in honor of {J}os\'{e}
  {A}dem ({O}axtepec, 1981)}, volume~12 of {\em Contemp. Math.}, pages
  289--292. Amer. Math. Soc., Providence, R.I., 1982.

\bibitem{RavenelWilson}
Douglas~C. Ravenel and W.~Stephen Wilson.
\newblock The {M}orava {$K$}-theories of {E}ilenberg-{M}ac {L}ane spaces and
  the {C}onner-{F}loyd conjecture.
\newblock {\em Amer. J. Math.}, 102(4):691--748, 1980.

\bibitem{Ribes}
Luis Ribes and Pavel Zalesskii.
\newblock {\em Profinite groups}.
\newblock Springer-Verlag, Berlin, 2000.

\bibitem{ShipleyHZ}
Brooke Shipley.
\newblock {$H\Bbb Z$}-algebra spectra are differential graded algebras.
\newblock {\em Amer. J. Math.}, 129(2):351--379, 2007.

\bibitem{Str}
N.~P. Strickland.
\newblock Gross-{H}opkins duality.
\newblock {\em Topology}, 39(5):1021--1033, 2000.

\bibitem{Wilson}
John~S. Wilson.
\newblock {\em Profinite groups}.
\newblock The Clarendon Press Oxford University Press, New York, 1998.

\end{thebibliography}

\end{document}